\documentclass[12pt]{article}

\usepackage[utf8]{inputenc}
\usepackage{t1enc}

\usepackage{amsmath, amsthm, amssymb, amsfonts, amscd}
\usepackage{mathtools}
\usepackage{mathrsfs}
\usepackage[mathscr]{eucal}
\usepackage{stmaryrd}
\usepackage{slashed}
\usepackage{thmtools}
\usepackage{accents}

\usepackage{graphicx, graphics, pict2e, epic, epstopdf}
\usepackage{xcolor}
\usepackage{tikz}
\usetikzlibrary{arrows.meta, positioning}
\usepackage{tikz-cd}
\usepackage{tkz-graph}

\usepackage{geometry}
\usepackage{setspace}
\usepackage{indentfirst}
\usepackage{booktabs}
\usepackage{float}
\usepackage{placeins}
\usepackage{subcaption}
\usepackage{enumerate}
\usepackage[symbol]{footmisc}
\usepackage{lastpage}
\usepackage{authblk}
\usepackage{verbatim}
\usepackage{lipsum}
\usepackage{cite}
\RequirePackage{doi}
\allowdisplaybreaks

\newcommand{\E}{\mathscr{E}}

\theoremstyle{plain}

\newtheorem{theorem}{Theorem}[section]

\newtheorem{corollary}[theorem]{Corollary}

\theoremstyle{remark}

\theoremstyle{definition}
\newtheorem{definition}[theorem]{Definition}

\usepackage{graphicx}

\makeatletter
\renewcommand{\maketitle}{
\begin{center}

{\Large\bfseries \@title\par}
\vspace{6mm}

{\large\bfseries \@author\par}
\vspace{4mm}

{\itshape \@address\par}
\vspace{2mm}

{\small\ttfamily \@email\par}
\vspace{5mm}

\vspace{5mm}

\end{center}
}
\makeatother

\makeatletter
\newcommand{\address}[1]{\gdef\@address{#1}}
\newcommand{\email}[1]{\gdef\@email{#1}}
\makeatother

\address{}
\email{}

\usepackage{fancyhdr}

\usepackage[margin=1cm,%
			font=footnotesize,%
			format=hang,%
			labelsep=period,%
			labelfont=bf]{caption}
\allowdisplaybreaks

\title{Graph Energies of Generalized and Shadow--Splitting Graphs}
\author{Ronak B. Dudhat$^{a}$, Vinodray J. Kaneria$^b$, Kalpesh M. Popat$^c$\footnote{Corresponding author.}}
\address{$^{a,b,c}$Department of Mathematics, \\ Saurashtra University,  Rajkot-360005, \\ Gujarat, India.}
\email{ronakdudhat108@gmail.com, kaneriavinodray@gmail.com,  kalpeshmpopat@gmail.com}

\date{\today}

\begin{document}

\maketitle
\thispagestyle{empty}

\begin{abstract}
We extend the notions of the $m$-splitting graph $\mathcal{S}_m(G)$ and the $m$-shadow graph $D_m(G)$ to introduce two new graph operations: the $(p,q)$-generalized splitting graph $\mathcal{S}_{p,q}(G)$ and the $(c,k)$-shadow-splitting graph $\mathcal{H}_{c,k}(G)$. We derive the adjacency energy of these constructions and as an application, identify several new infinite families of equienergetic and borderenergetic graphs.

\vspace{2mm}

\vspace{1mm}
\end{abstract}

\onehalfspacing

\section{Introduction}

Let $G$ be a simple graph with vertex set $V(G)=\{v_1,v_2,\dots,v_n\}$. The adjacency matrix of $G$ is the $n\times n$ matrix $A(G)=[a_{ij}]$, where $a_{ij}=1$ if $v_i$ is adjacent to $v_j$, and $0$ otherwise. Let $\lambda_1, \lambda_2, \dots, \lambda_n$ denote the eigenvalues of $A(G)$; together they constitute the spectrum of $G$. The energy of $G$, denoted by $\E(G)$, is defined as the sum of the absolute values of its eigenvalues, that is,
\[
\E(G)=\sum_{i=1}^n |\lambda_i|.
\]
The concept of graph energy was pioneered by Gutman \cite{Gutman1978} in 1978, originating from H\"uckel molecular orbital theory, where it is used to approximate the total $\pi$-electron energy of conjugated hydrocarbons. In the decades since its introduction, graph energy has emerged as a fundamental topic in spectral graph theory, inspiring the study of various energy-like invariants and complex graph constructions. For a detailed treatment of the subject, we refer the reader to the monographs \cite{li2012graph, brouwer2012, cvetkovic2010}.

Two graphs $G$ and $H$ of the same order are said to be \emph{equienergetic} if $\E(G)=\E(H)$. While cospectral graphs are trivially equienergetic, the construction of non-cospectral equienergetic graphs remains a central area of investigation \cite{Gutman2015,Ramane2004,bran2004,mili2009,vaidya2020,vaidya2024,kalp2022}. Furthermore, a graph $G$ of order $n$ is called \emph{borderenergetic} if its energy equals that of the complete graph, that is, $\E(G) = 2(n-1)$.  The concept of borderenergetic graphs was formally introduced by Gong et al. \cite{gong2015borderenergetic}. Their work motivated a surge of interest in discovering new non-complete borderenergetic graph families, particularly through algebraic techniques and constructive approaches \cite{sha2016, fur2017, tura2017,  deng1, hou2017, deng2, elu2018, dede2024, dedemadem2024}.

Graph operations are instrumental in generating new families of graphs with predictable spectral properties. Significant developments in this direction include the introduction of the $m$-splitting graph $\mathcal{S}_m(G)$ and the $m$-shadow graph $D_m(G)$  \cite{Abdel2013}. Vaidya and Popat \cite{vaidyapopat2017} established foundational results for these constructions, proving the identities $\E(\mathcal{S}_m(G)) = \sqrt{1+4m}\,\E(G)$  and $\E(D_m(G)) = m\,\E(G)$. These results demonstrated that the energy of the resulting graphs remains a simple scaling of the energy of the base graph, providing a systematic way to construct equienergetic and borderenergetic graphs. In the present work, we extend these results by introducing two new graph operators: 
\begin{enumerate}
    \item The ${(p,q)}$-\textbf{Generalized Splitting} graph, denoted by $\mathcal{S}_{p,q}(G)$, which extends the $m$-splitting framework by incorporating $p$ copies of the base graph and $q$ sets of splitting vertices.
    \item The $(c,k)$-\textbf{Shadow-Splitting} graph, denoted by $\mathcal{H}_{c,k}(G)$, which combines the structural features of shadow and splitting graphs into a unified block-matrix construction.
\end{enumerate}

We derive explicit formulas for the energy of these constructions in terms of the energy of the base graph $G$. Furthermore, we demonstrate the effectiveness of these operators by constructing several new infinite families of non-cospectral equienergetic graphs and determining parameter values that yield borderenergetic graphs. We conclude this section by introducing the notation and preliminary results that will be used throughout the remainder of the paper.
\section{Notation and Preliminaries} 
 Throughout this paper, we consider $G$ to be a simple and undirected graph of order $n$. Let $\operatorname{Spec}(M)$ denote the multiset of eigenvalues of a matrix $M$. Let \(\operatorname{rank}(M)\) denote the rank of a matrix \(M\). A partition 
\(\mathcal{P}=\{C_1,C_2,\dots,C_k\}\) of the index set of \(M\) is called an 
\emph{equitable partition} if, for every \(i,j\in\{1,2,\dots,k\}\), the sum of 
the entries in each row of the block \(M_{ij}\) is a constant \(q_{ij}\), 
independent of the choice of the row within \(C_i\). The \emph{quotient matrix} 
\(Q=(q_{ij})\) is the \(k\times k\) matrix whose entries are these common row 
sums. A fundamental result\cite{brouwer2012} asserts that \(\operatorname{Spec}(Q)\) is a 
submultiset of \(\operatorname{Spec}(M)\) . We denote the identity matrix of order $n$ by $I_n$, the $m \times n$ matrix of all ones by $J_{m \times n}$ and the $n \times 1$ column vector of all ones by $\mathbf{1}_n$. The complete graph of order $n$ is denoted by $K_n$, and the complete bipartite graph with partitions of size $m$ and $n$ is denoted by $K_{m,n}$. It is well known that the energy of the complete graph is $\E(K_n) = 2(n-1)$, and the energy of the complete bipartite graph is $\E(K_{m,n}) = 2\sqrt{mn}$. For a vertex $v$ in a graph $G$, the set $N(v)$ is defined as the set of all vertices in $G$ that are adjacent to $v$. The \emph{$m$-Splitting graph} $\mathcal{S}_m(G)$ of a graph $G$ is obtained by adding $m$ new vertices $v_1, v_2, \dots, v_m$ for each vertex $v \in V(G)$, such that each $v_i$ (for $1 \le i \le m$) is adjacent to every vertex in the neighborhood $N(v)$ of $G$. The \emph{$m$-Shadow graph} $D_m(G)$ of a graph $G$ is constructed by taking $m$ copies of $G$, denoted by $G_1, G_2, \dots, G_m$, and joining each vertex $u \in V(G_i)$ to the neighbors of the corresponding vertex $v \in V(G_j)$ for all $1 \le i, j \le m$. For any two matrices $A = [a_{ij}]$ and $B$, the product $A \otimes B$ is defined as the block matrix $[a_{ij}B]$. If $\{\lambda_i\}$ and $\{\mu_j\}$ are the eigenvalues of $A$ and $B$, respectively, then it is a well-known property~\cite{Horn1991} that the eigenvalues of $A \otimes B$ are the products $\{\lambda_i \mu_j\}$.  The \emph{Kronecker product} of two graphs $G$ and $H$, denoted by $G \otimes H$, is the graph with vertex set $V(G) \times V(H)$ where two vertices $(u_1, v_1)$ and $(u_2, v_2)$ are adjacent in $G \otimes H$ if and only if $u_1$ is adjacent to $u_2$ in $G$ and $v_1$ is adjacent to $v_2$ in $H$. A well-known property of this construction is that the adjacency matrix of the product graph is the Kronecker product of the adjacency matrices of the factor graphs, specifically $A(G \otimes H) = A(G) \otimes A(H)$. Consequently, the energy of the Kronecker product is multiplicative, satisfying a known identity~\cite{Balakri2004} which states that $\mathcal{E}(G \otimes H) = \mathcal{E}(G)\mathcal{E}(H)$.
\section{Energy of the \texorpdfstring{$(p,q)$}{(p,q)}-Generalized  Splitting Graph}
In this section, we define the $(p,q)$-\emph{Generalized Splitting} graph and derive its energy. This construction generalizes the $m$-\emph{Splitting} graph by allowing multiple copies of the base graph.

\begin{definition}
    Let $G$ be a graph of order $n$. For $p,q\ge 1$ the \emph{(p,q)-Generalized Splitting} graph, $\mathcal{S}_{p,q}(G)$, is formed by taking $p$ disjoint copies of $G$, as $G^{(1)}, \dots, G^{(p)}$ (write $v_i^{(\ell)}$ for the copy of $v_i$ in $G^{(\ell)}$) and adding $q$ sets of isolated vertices as $U^{(k)}=\{u_1^{(k)},\dots,u_n^{(k)}\},\,k=1,\dots,q$ such that for every $i=1,\dots,n$
$$N(u_i^{(k)}) = \bigcup_{\ell=1}^p N_{G^{(\ell)}}(v_i^{(\ell)}),\;k=1,\dots,q.$$
\end{definition} 
\noindent The following figure illustrates the $(2,2)$-generalized splitting graph $\mathcal{S}_{2,2}(C_4)$:\\
\FloatBarrier 
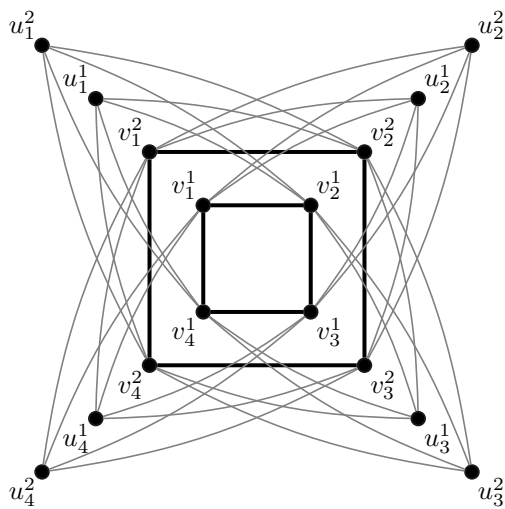
\begin{figure}[!h] 
\centering
\begin{tikzpicture}[scale=1.0, rotate=45, font=\footnotesize\bfseries,
    v_base/.style={circle, draw=black!90, fill=black, inner sep=1.8pt, line width=0.7pt},
    v_copy/.style={circle, draw=black!90, fill=black, inner sep=1.8pt, line width=0.7pt},
    v_split/.style={circle, draw=black!90, fill=black, inner sep=1.8pt, line width=0.7pt},
    edge_G/.style={line width=1.5pt, draw=black}, 
    edge_U/.style={line width=0.6pt, draw=black!50}       
]

\node[v_base] (a1) at (0,1) {};
\node[v_base] (a2) at (1,0) {};
\node[v_base] (a3) at (0,-1) {};
\node[v_base] (a4) at (-1,0) {};

\draw[edge_G] (a1)--(a2)--(a3)--(a4)--(a1)--cycle;

\node[anchor=south east, inner sep=2pt] at (a1) {$v_1^1$};
\node[anchor=south west, inner sep=2pt] at (a2) {$v_2^1$};
\node[anchor=north west, inner sep=2pt] at (a3) {$v_3^1$};
\node[anchor=north east, inner sep=2pt] at (a4) {$v_4^1$};

\node[v_copy] (b1) at (0,2) {};
\node[v_copy] (b2) at (2,0) {};
\node[v_copy] (b3) at (0,-2) {};
\node[v_copy] (b4) at (-2,0) {};

\draw[edge_G] (b1)--(b2)--(b3)--(b4)--(b1)--cycle;

\node[anchor=south east, inner sep=2pt] at (b1) {$v_1^2$};
\node[anchor=south west, inner sep=2pt] at (b2) {$v_2^2$};
\node[anchor=north west, inner sep=2pt] at (b3) {$v_3^2$};
\node[anchor=north east, inner sep=2pt] at (b4) {$v_4^2$};

\node[v_split] (u11) at (0,3) {};
\node[v_split] (u21) at (3,0) {};
\node[v_split] (u31) at (0,-3) {};
\node[v_split] (u41) at (-3,0) {};

\node[v_split] (u12) at (0,4) {};
\node[v_split] (u22) at (4,0) {};
\node[v_split] (u32) at (0,-4) {};
\node[v_split] (u42) at (-4,0) {};

\node[anchor=south east, inner sep=2pt] at (u11) {$u_1^1$};
\node[anchor=south west, inner sep=2pt] at (u21) {$u_2^1$};
\node[anchor=north west, inner sep=2pt] at (u31) {$u_3^1$};
\node[anchor=north east, inner sep=2pt] at (u41) {$u_4^1$};

\node[anchor=south east, inner sep=2pt] at (u12) {$u_1^2$};
\node[anchor=south west, inner sep=2pt] at (u22) {$u_2^2$};
\node[anchor=north west, inner sep=2pt] at (u32) {$u_3^2$};
\node[anchor=north east, inner sep=2pt] at (u42) {$u_4^2$};

\foreach \i in {1,2}{
    \draw[edge_U] (u1\i) to[bend left=10] (a2);
    \draw[edge_U] (u1\i) to[bend right=10] (a4);
    \draw[edge_U] (u1\i) to[bend left=10] (b2);
    \draw[edge_U] (u1\i) to[bend right=10] (b4);
}

\foreach \i in {1,2}{
    \draw[edge_U] (u2\i) to[bend right=10] (a1);
    \draw[edge_U] (u2\i) to[bend left=10] (a3);
    \draw[edge_U] (u2\i) to[bend right=10] (b1);
    \draw[edge_U] (u2\i) to[bend left=10] (b3);
}

\foreach \i in {1,2}{
    \draw[edge_U] (u3\i) to[bend right=10] (a2);
    \draw[edge_U] (u3\i) to[bend left=10] (a4);
    \draw[edge_U] (u3\i) to[bend right=10] (b2);
    \draw[edge_U] (u3\i) to[bend left=10] (b4);
}

\foreach \i in {1,2}{
    \draw[edge_U] (u4\i) to[bend left=10] (a1);
    \draw[edge_U] (u4\i) to[bend right=10] (a3);
    \draw[edge_U] (u4\i) to[bend left=10] (b1);
    \draw[edge_U] (u4\i) to[bend right=10] (b3);
}

\end{tikzpicture}

\caption{The $(2,2)$-generalized splitting graph \(\mathcal{S}_{2,2}(C_4)\).}
\end{figure}  \FloatBarrier 
\noindent It is worth noting that for $p=1$ and $q=m$, the generalized splitting graph $\mathcal{S}_{p,q}(G)$ reduces to the $m$-splitting graph $\mathcal{S}_m(G)$. The adjacency matrix of $\mathcal{S}_{p,q}(G)$ is given by the block matrix of order $(p+q)n$:
\[
A(\mathcal{S}_{p,q}(G)) = 
\begin{bmatrix}
I_p \otimes A(G) & J_{p \times q} \otimes A(G) \\
J_{q \times p} \otimes A(G) & 0_{qn \times q n}
\end{bmatrix}
=
\begin{bmatrix}
I_p & J_{p \times q} \\
J_{q \times p} & 0_q
\end{bmatrix} \otimes A(G).
\]
\begin{theorem}
For any $p, q \ge 1$, the energy of the $(p,q)$-generalized splitting graph is
\[
\E(\mathcal{S}_{p,q}(G)) = \big( p - 1 + \sqrt{1 + 4pq} \big) \E(G).
\]
\end{theorem}
\begin{proof}
Let $M = \begin{bmatrix} I_p & J_{p \times q} \\ J_{q \times p} & 0_q \end{bmatrix}$ be the $(p+q) \times (p+q)$ matrix. The eigenvalues of $A(\mathcal{S}_{p,q}(G))$ are the products of the eigenvalues of $M$ and $A(G)$.
A straightforward computation shows that the characteristic polynomial of $M$ is
\[
\det(\mu I_{p+q}-M)
=\mu^{q-1}(\mu-1)^{p-1}(\mu^2-\mu-pq).
\]
It follows that the spectrum of $M$ is
\[
\underbrace{1,\dots,1}_{p-1},\quad
\underbrace{0,\dots,0}_{q-1},\quad
\frac{1\pm\sqrt{1+4pq}}{2}.
\]
The eigenvalues of $\mathcal{S}_{p,q}(G)$ are of the form $\mu \lambda_i$, where $\mu \in Spec(M)$ and $\lambda_i \in Spec(A(G))$. Utilizing the multiplicative property of energy for Kronecker products, we have:
\[
\E\big(\mathcal{S}_{p,q}(G)\big) = \left( \sum_{\mu \in Spec(M)} |\mu| \right) \sum_{i=1}^{n} |\lambda_i|.
\]
For the matrix $M$, the sum of the absolute values of the eigenvalues is:
\begin{align*}
\sum_{\mu \in Spec(M)} |\mu| &= (p-1)|1| + (q-1)|0| + \left| \frac{1 + \sqrt{1 + 4pq}}{2} \right| + \left| \frac{1 - \sqrt{1 + 4pq}}{2} \right| \\
&= (p-1) + \frac{1 + \sqrt{1 + 4pq}}{2} + \frac{\sqrt{1 + 4pq} - 1}{2} \\
&= (p-1) + \sqrt{1 + 4pq}.
\end{align*}
Substituting this into the energy expression for the graph, we obtain:
\[
\E\big(\mathcal{S}_{p,q}(G)\big) = \left( p - 1 + \sqrt{1 + 4pq} \right) \E(G).
\] 
\end{proof}
\section{Energy of the \texorpdfstring{$(c,k)$}{(c,k)}-Shadow-Splitting Graph}
In this section, we introduce the $(c,k)$-\emph{Shadow-Splitting} graph, which combines the structural properties of shadow and splitting graphs. We derive its energy as a function of the energy of the base graph $G$.
\begin{definition}
Let $G$ be a graph of order $n$. For $c,k\ge 1$ the $(c,k)$-\emph{Shadow-Splitting} graph, $\mathcal{H}_{c,k}(G)$, is formed by taking $c$ disjoint copies of $G$, as $G^{(1)}, \dots, G^{(c)}$ (write $v_i^{(\ell)}$ for the copy of $v_i$ in $G^{(\ell)}$) and adding $k$ sets of isolated vertices as $U^{(r)}=\{u_1^{(r)},\dots,u_n^{(r)}\},\,r=1,\dots,k$ such that for every $i=1,\dots,n$:
$$N(u_i^{(r)}) = \bigcup_{v_j \in N_G(v_i)} \{v_j^{(1)}, \dots, v_j^{(c)}\}, \; \text{for } r=1,\dots,k$$ and
$$N(v_i^{(s)}) = \bigcup_{v_j \in N_G(v_i)} \left( \{v_j^{(1)}, \dots, v_j^{(c)}\} \cup \{u_j^{(1)}, \dots, u_j^{(k)}\} \right), \;\text{for } s=1,\dots,c.$$
\end{definition} 
\noindent The following figure illustrates the $(2,2)$-shadow-splitting graph $\mathcal{H}_{2,2}(C_4)$:\\
\FloatBarrier
\begin{figure}[!h]
\centering

\begin{tikzpicture}[scale=1.1, rotate=45, font=\footnotesize\bfseries,
    v_base/.style={circle, draw=black!90, fill=black, inner sep=1.8pt, line width=0.7pt},
    v_shadow/.style={circle, draw=black!90, fill=black, inner sep=1.8pt, line width=0.7pt},
    v_split/.style={circle, draw=black!90, fill=black, inner sep=1.8pt, line width=0.7pt},
    edge_G/.style={line width=1.5pt, draw=black}, 
    edge_S/.style={line width=1.1pt, draw=black}, 
    edge_U/.style={line width=0.6pt, draw=black!50}       
]

\node[v_base] (a1) at (0,1) {};
\node[v_base] (a2) at (1,0) {};
\node[v_base] (a3) at (0,-1) {};
\node[v_base] (a4) at (-1,0) {};

\node[anchor=south east, inner sep=2pt] at (a1) {$v_1^1$};
\node[anchor=south west, inner sep=2pt] at (a2) {$v_2^1$};
\node[anchor=north west, inner sep=2pt] at (a3) {$v_3^1$};
\node[anchor=north east, inner sep=2pt] at (a4) {$v_4^1$};

\node[v_shadow] (b1) at (0,2) {};
\node[v_shadow] (b2) at (2,0) {};
\node[v_shadow] (b3) at (0,-2) {};
\node[v_shadow] (b4) at (-2,0) {};

\node[anchor=south east, inner sep=2pt] at (b1) {$v_1^2$};
\node[anchor=south west, inner sep=2pt] at (b2) {$v_2^2$};
\node[anchor=north west, inner sep=2pt] at (b3) {$v_3^2$};
\node[anchor=north east, inner sep=2pt] at (b4) {$v_4^2$};

\node[v_split] (u11) at (0,3) {};
\node[v_split] (u21) at (3,0) {};
\node[v_split] (u31) at (0,-3) {};
\node[v_split] (u41) at (-3,0) {};

\node[v_split] (u12) at (0,4) {};
\node[v_split] (u22) at (4,0) {};
\node[v_split] (u32) at (0,-4) {};
\node[v_split] (u42) at (-4,0) {};

\node[anchor=south east, inner sep=2pt] at (u11) {$u_1^1$};
\node[anchor=south west, inner sep=2pt] at (u21) {$u_2^1$};
\node[anchor=north west, inner sep=2pt] at (u31) {$u_3^1$};
\node[anchor=north east, inner sep=2pt] at (u41) {$u_4^1$};

\node[anchor=south east, inner sep=2pt] at (u12) {$u_1^2$};
\node[anchor=south west, inner sep=2pt] at (u22) {$u_2^2$};
\node[anchor=north west, inner sep=2pt] at (u32) {$u_3^2$};
\node[anchor=north east, inner sep=2pt] at (u42) {$u_4^2$};


\draw[edge_G] (a1)--(a2)--(a3)--(a4)--(a1)--cycle;
\draw[edge_G] (b1)--(b2)--(b3)--(b4)--(b1)--cycle;

\draw[edge_S] (b1)--(a2)--(b3)--(a4)--(b1)--cycle;
\draw[edge_S] (b2)--(a3)--(b4)--(a1)--(b2)--cycle;


\foreach \i in {1,2}{
    \draw[edge_U] (u1\i) to[bend left=10] (a2);
    \draw[edge_U] (u1\i) to[bend right=10] (a4);
    \draw[edge_U] (u1\i) to[bend left=10] (b2);
    \draw[edge_U] (u1\i) to[bend right=10] (b4);
}

\foreach \i in {1,2}{
    \draw[edge_U] (u2\i) to[bend right=10] (a1);
    \draw[edge_U] (u2\i) to[bend left=10] (a3);
    \draw[edge_U] (u2\i) to[bend right=10] (b1);
    \draw[edge_U] (u2\i) to[bend left=10] (b3);
}

\foreach \i in {1,2}{
    \draw[edge_U] (u3\i) to[bend right=10] (a2);
    \draw[edge_U] (u3\i) to[bend left=10] (a4);
    \draw[edge_U] (u3\i) to[bend right=10] (b2);
    \draw[edge_U] (u3\i) to[bend left=10] (b4);
}

\foreach \i in {1,2}{
    \draw[edge_U] (u4\i) to[bend left=10] (a1);
    \draw[edge_U] (u4\i) to[bend right=10] (a3);
    \draw[edge_U] (u4\i) to[bend left=10] (b1);
    \draw[edge_U] (u4\i) to[bend right=10] (b3);
}

\end{tikzpicture}

\caption{The $(2,2)$-shadow-splitting graph \(\mathcal{H}_{2,2}(C_4)\).}
\end{figure}
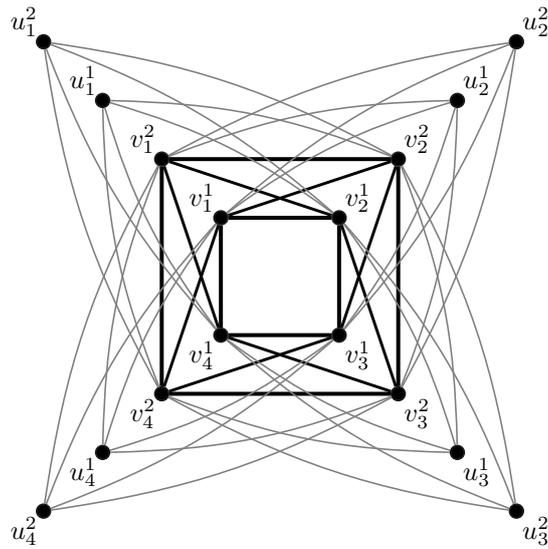 \FloatBarrier 
\noindent Additionally, it is observed that for $c=1$ and $k=m$, the shadow-splitting graph $\mathcal{H}_{c,k}(G)$ reduces to the $m$-splitting graph $\mathcal{S}_m(G)$. The adjacency matrix of $\mathcal{H}_{c,k}(G)$ of order $(c+k)n$ is given by:
\[
A(\mathcal{H}_{c,k}(G)) = 
\begin{bmatrix}
J_c \otimes A(G) & J_{c \times k} \otimes A(G) \\
J_{k \times c} \otimes A(G) & 0_{kn \times k n}
\end{bmatrix}
=
\begin{bmatrix}
J_c & J_{c \times k} \\
J_{k \times c} & 0_k
\end{bmatrix} \otimes A(G).
\]

\begin{theorem}
For any positive integers $c,k \ge 1$, the energy of the $(c,k)$-shadow-splitting graph is
\[
\E(\mathcal{H}_{c,k}(G)) = \sqrt{c^2 + 4ck} \; \E(G).
\]
\end{theorem}
\begin{proof}
Let $M = \begin{bmatrix} J_c & J_{c \times k} \\ J_{k \times c} & 0_k \end{bmatrix}$ be the $(c+k) \times (c+k)$ matrix. The eigenvalues of $\mathcal{H}_{c,k}(G)$ are the products of the eigenvalues of $M$ and $A(G)$.

Observe that all rows in the first $c$ blocks of $M$ are identical, and all rows in the last $k$ blocks are identical. Since these two row vectors are linearly independent, $rank(M) = 2$. Consequently, $M$ has $c+k-2$ eigenvalues equal to zero. The non-zero eigenvalues of $M$ are the roots of the characteristic equation of the quotient matrix $\tilde{M} = \begin{bmatrix} c & k \\ c & 0 \end{bmatrix}$, given by $\theta^2 - c\theta - ck = 0.$ Therefore, the non-zero eigenvalues are \[\theta_{1,2} = \frac{c \pm \sqrt{c^2 + 4ck}}{2}.\]
The sum of the absolute values of its eigenvalues:
\begin{align*}
\E(M) &= \left| \frac{c + \sqrt{c^2 + 4ck}}{2} \right| + \left| \frac{c - \sqrt{c^2 + 4ck}}{2} \right| \\
&= \frac{c + \sqrt{c^2 + 4ck}}{2} + \frac{\sqrt{c^2 + 4ck} - c}{2} \\
&= \sqrt{c^2 + 4ck}.
\end{align*}
Since $A(\mathcal{H}_{c,k}(G))=M \otimes A(G)
$, we have:
\[
\E(\mathcal{H}_{c,k}(G)) = \sqrt{c^2 + 4ck} \; \E(G).
\]
\end{proof}
\section{Construction of Equienergetic Graphs}
In this section, we utilize the energy identities for the $(p,q)$-generalized splitting graph and the $(c,k)$-shadow-splitting graph to construct infinite families of non-isomorphic equienergetic graphs.
\begin{corollary}
Let $G_1$ and $G_2$ be equienergetic graphs of the same order. Then, for any positive integers $p,q,c,k \ge 1$, the following pairs of graphs are equienergetic:
\begin{enumerate}
    \item $\mathcal{S}_{p,q}(G_1)$ and $\mathcal{S}_{p,q}(G_2)$,
    \item $\mathcal{H}_{c,k}(G_1)$ and $\mathcal{H}_{c,k}(G_2)$.
\end{enumerate}
\end{corollary}
\begin{proof}
Since $G_1$ and $G_2$ are equienergetic, we have
\[
\E(G_1)=\E(G_2).
\]
Using the energy formula for the $(p,q)$-generalized splitting graph, we obtain
\[
\E\big(\mathcal{S}_{p,q}(G_1)\big)
= \big(p-1+\sqrt{1+4pq}\big)\E(G_1)
\]
and
\[
\E\big(\mathcal{S}_{p,q}(G_2)\big)
= \big(p-1+\sqrt{1+4pq}\big)\E(G_2).
\]
Therefore,
\[
\E\big(\mathcal{S}_{p,q}(G_1)\big)=\E\big(\mathcal{S}_{p,q}(G_2)\big).
\]
Similarly, by the energy formula for the $(c,k)$-shadow-splitting graph, we have
\[
\E\big(\mathcal{H}_{c,k}(G_1)\big)
= \sqrt{c^2+4ck}\,\E(G_1)
= \sqrt{c^2+4ck}\,\E(G_2)
= \E\big(\mathcal{H}_{c,k}(G_2)\big).
\]
\end{proof}
\begin{corollary}
For $t,m \ge 1$ and $k \in \{1, -1\}$ the graphs $\mathcal{S}_{p_1, q_1}(G)$ and $\mathcal{S}_{p_2, q_2}(G)$ are equienergetic, where $p_1 = (5t-2)^2 m + k(5t-2)$,\; $q_1 = m$,\; $p_2 = 5t^2 m + k t$ and $q_2 = 5(2t-1)^2 m + k(4t-2)$.
\end{corollary}
\begin{proof}
To establish that these graphs are equienergetic, we first verify that they have the same order by examining the sum of their parameters. Direct expansion yields
$p_1+q_1= (5t-2)^2m + k(5t-2) + m 
= (25t^2 - 20t + 5)m + 5kt - 2k,
$
and
$
p_2 + q_2 = 5t^2m + kt + 5(2t-1)^2m + k(4t-2) 
= (25t^2 - 20t + 5)m + 5kt - 2k.
$
Since $p_1 + q_1 = p_2 + q_2$, the two graphs have an equal number of vertices.  The energy of $\mathcal{S}_{p_1, q_1}(G)$ is:
\begin{align*}
\mathscr{E}(\mathcal{S}_{p_1, q_1}(G)) &= \left( \sqrt{1+4p_1 q_1} + p_1 - 1 \right) \mathscr{E}(G) \\
&= \left( \sqrt{1+4((5t-2)^2 m + k(5t-2))m} + (5t-2)^2 m + k(5t-2) - 1 \right) \mathscr{E}(G) \\
&=\left(\sqrt{k^2 + 4(5t-2)^2m^2 + 4km(5t-2)}  + (5t-2)^2 m + k(5t-2) - 1\right)\mathscr{E}(G) \\
&= \left( \sqrt{(2(5t-2)m + k)^2} + (25t^2 - 20t + 4)m + 5kt - 2k - 1 \right) \mathscr{E}(G) \\
&= \left((10t-4)m + k + 25t^2 m - 20tm + 4m + 5kt - 2k - 1\right) \mathscr{E}(G) \\
&= \left((25t^2 - 10t)m + 5kt - k - 1\right) \mathscr{E}(G).
\end{align*}
The energy of $\mathcal{S}_{p_2, q_2}(G)$ is:
\begin{align*}
\mathscr{E}(\mathcal{S}_{p_2, q_2}(G)) &= \left( \sqrt{1+4p_2 q_2} + p_2 - 1 \right) \mathscr{E}(G) \\
&= \left( \sqrt{1+4(5t^2 m + kt)(5(2t-1)^2 m + k(4t-2))} + 5t^2 m + kt - 1 \right) \mathscr{E}(G) \\
&= \left( \sqrt{(10t(2t-1)m + k(4t-1))^2} + 5t^2 m + kt - 1 \right) \mathscr{E}(G) \\
&= \left(20t^2 m - 10tm + 4kt - k + 5t^2 m + kt - 1\right) \mathscr{E}(G) \\
&= \left((25t^2 - 10t)m + 5kt - k - 1\right) \mathscr{E}(G).
\end{align*}
Therefore,
\[
\E\big(\mathcal{S}_{p_1,q_1}(G)\big)=\E\big(\mathcal{S}_{p_2,q_2}(G)\big),
\]
\end{proof}
\begin{corollary}
For $m > t \ge 1$, the graphs $\mathcal{H}_{m+t, 2m-t}(G)$ and $\mathcal{H}_{3m-t, t}(G)$ are equienergetic.
\end{corollary}

\begin{proof}
Let $H_1 = \mathcal{H}_{m+t, 2m-t}(G)$ and $H_2 = \mathcal{H}_{3m-t, t}(G)$. We first observe that both graphs have the same order. Indeed, the number of vertices in $H_1$ is given by $((m+t) + (2m-t))n = 3mn$, which is identical to the order of $H_2$ given by $((3m-t) + t)n = 3mn$. The energy of $H_1$ is calculated as
\begin{align*}
\E(H_1) &= \sqrt{(m+t)^2 + 4(m+t)(2m-t)}\; \E(G) \\
&= \sqrt{m^2 + 2mt + t^2 + 8m^2 + 4mt - 4t^2}\;\E(G) \\
&= \sqrt{9m^2 + 6mt - 3t^2}\; \E(G).
\end{align*}
Similarly, for the graph $H_2$, the energy is 
\begin{align*}
\E(H_2) &= \sqrt{(3m-t)^2 + 4(3m-t)t}\; \E(G) \\
&= \sqrt{9m^2 - 6mt + t^2 + 12mt - 4t^2}\; \E(G) \\
&= \sqrt{9m^2 + 6mt - 3t^2}\; \E(G).
\end{align*}
Since $\E(H_1) = \E(H_2)$ and the parameters $c_1 = m+t$ and $c_2 = 3m-t$ are distinct for $m > t$, the two graphs are non-isomorphic. We thus conclude that $H_1$ and $H_2$ are equienergetic.
\end{proof}
\begin{corollary}
The generalized splitting graph $\mathcal{S}_{p,q}(G)$ is equienergetic with the shadow graph $D_{p+q}(G)$ if and only if $q = 4p - 2$.
\end{corollary}
\begin{proof}
Recall that the energy of the $(p+q)$-shadow graph is given by $\E(D_{p+q}(G)) = (p+q)\E(G)$. Based on the energy formula for $\mathcal{S}_{p,q}(G)$, the two graphs are equienergetic if and only if
\[
\sqrt{1 + 4pq} + p - 1 = p + q \Rightarrow \sqrt{1 + 4pq} = q + 1.
\]
Squaring both sides of this expression yields $1 + 4pq = q^2 + 2q + 1$. Rearranging the terms, we have $4pq = q^2 + 2q$, and dividing by $q$ to find $4p = q + 2$. This implies that the condition for equienergeticity is $q = 4p - 2$.
\end{proof}
\begin{corollary}
For $c \ge 2$ and $k \ge 1$, the shadow-splitting graph $\mathcal{H}_{c,k}(G)$ and the shadow graph $D_{c+k}(G)$ are equienergetic if and only if $k = 2c$.
\end{corollary}
\begin{proof}
Utilizing the energy formula for the $(c+k)$-shadow graph, $\E(D_{c+k}(G)) = (c+k)\E(G)$, the two graphs are equienergetic if and only if
\[
\sqrt{c^2 + 4ck} = c + k.
\]
By squaring both sides of this identity, we obtain the quadratic expression $c^2 + 4ck = c^2 + 2ck + k^2$. This simplifies directly to $2ck = k^2$, and dividing by $k$ to yield $2c = k$.
\end{proof}
\begin{corollary}
The generalized splitting graph $\mathcal{S}_{2,1}(G)$ is equienergetic with the Kronecker product $G \otimes K_3$.
\end{corollary}
\begin{proof}
From the energy formula for the generalized splitting graph, setting $p=2$ and $q=1$ yields 
\[
\E(\mathcal{S}_{2,1}(G)) = \left( \sqrt{1 + 4 \cdot 2 \cdot 1} + 2 - 1 \right) \E(G) = 4 \E(G).
\]
We now consider the complete graph $K_3$. It is well known that $\E(K_3) = 4$. Based on the properties of the Kronecker product, specifically that $\E(A \otimes B) = \E(A) \E(B)$, it follows that
\[
\E(G \otimes K_3) = \E(G) \E(K_3) = 4 \E(G).
\]
Since $\E(\mathcal{S}_{2,1}(G)) = \E(G \otimes K_3)$ and both graphs are of order $3n$, we conclude that they are equienergetic.
\end{proof}
\begin{corollary}
For $m \ge 1$, the generalized splitting graph $\mathcal{S}_{2m, 8m-2}(G)$ is equienergetic with the Kronecker product $G \otimes K_{5m-1, 5m-1}$.
\end{corollary}
\begin{proof}
Let $H_1 = \mathcal{S}_{2m, 8m-2}(G)$ and $H_2 = G \otimes K_{5m-1, 5m-1}$. We first verify that the orders of both graphs are identical. The order of $H_1$, determined by the parameters $p=2m$ and $q=8m-2$, is given by $N_1 = (2m + 8m - 2)n = (10m-2)n$. For the Kronecker product $H_2$, the order is the product of the orders of the factor graphs, yielding $N_2 = (5m-1 + 5m-1)n = (10m-2)n$. Substituting the parameters $p$ and $q$ into the energy formula for the generalized splitting graph, we find that
\begin{align*}
\E(H_1) &= \left( \sqrt{1+4(2m)(8m-2)} + 2m - 1 \right) \E(G) \\
&= \left( \sqrt{1+64m^2-16m} + 2m - 1 \right) \E(G) \\
&= \left( \sqrt{(8m-1)^2} + 2m - 1 \right) \E(G) \\
&= (8m - 1 + 2m - 1) \E(G) \\
&= (10m - 2) \E(G).
\end{align*}
For the Kronecker product $H_2$, we utilize the multiplicative property of graph energy. we have
\begin{align*}
\E(H_2) &= \E(K_{5m-1, 5m-1}) \E(G) \\
&= 2\sqrt{(5m-1)(5m-1)}\; \E(G) \\
&= 2(5m-1) \E(G) \\
&= (10m-2) \E(G).
\end{align*}
Since the energy values are identical and both graphs share the same order, we conclude that $\E(H_1) = \E(H_2)$ for all $m \ge 1$.
\end{proof}
\begin{corollary}
For $m \ge 1$ , the graphs $\mathcal{S}_{3m+1, 12m+2}(G)$ and $\mathcal{H}_{5m+1, 10m+2}(G)$ are equienergetic.
\end{corollary}
\begin{proof}
Let $H_1 = \mathcal{S}_{3m+1, 12m+2}(G)$ and $H_2 = \mathcal{H}_{5m+1, 10m+2}(G)$. The order of $H_1$ is given by $((3m+1) + (12m+2))n = (15m+3)n$, while the order of $H_2$ is $((5m+1) + (10m+2))n = (15m+3)n$. The energy for $H_1$ is calculated as
\begin{align*}
\E(H_1) &= \left( \sqrt{1 + 4(3m+1)(12m+2)} + (3m+1) - 1 \right) \E(G) \\
&= \left( \sqrt{144m^2 + 72m + 9} + 3m \right) \E(G) \\
&= \left( \sqrt{(12m+3)^2} + 3m \right) \E(G) \\
&= (12m + 3 + 3m) \E(G) \\
&= (15m + 3) \E(G).
\end{align*}
The energy for $H_2$ is given by
\begin{align*}
\E(H_2) &= \sqrt{(5m+1)^2 + 4(5m+1)(10m+2)}\; \E(G) \\
&= \sqrt{25m^2 + 10m + 1 + 200m^2 + 80m + 8}\; \E(G) \\
&= \sqrt{225m^2 + 90m + 9}\; \E(G) \\
&= \sqrt{(15m+3)^2}\; \E(G) \\
&= (15m + 3)\; \E(G).
\end{align*}
Since $\E(H_1) = \E(H_2)$ and both graphs share the same order, they are equienergetic for all $m \ge 1$.
\end{proof}
\begin{corollary}
For $t \ge 1$, the graphs $\mathcal{H}_{10t-4, 20t-8}(G)$, $D_{30t-12}(G)$, $K_{15t-6, 15t-6} \otimes G$, and $\mathcal{S}_{6t-2, 24t-10}(G)$ are mutually equienergetic.
\end{corollary}
\begin{proof}
We first establish that all four graphs share the same order $N = (30t-12)n$. The order of $H_1 = \mathcal{H}_{10t-4, 20t-8}(G)$ is $((10t-4) + (20t-8))n = (30t-12)n$. The shadow graph $H_2 = D_{30t-12}(G)$ has order $(30t-12)n$ by definition. For the Kronecker product $H_3 = K_{15t-6, 15t-6} \otimes G$, the order is $(15t-6 + 15t-6)n = (30t-12)n$. Finally, the generalized splitting graph $H_4 = \mathcal{S}_{6t-2, 24t-10}(G)$ has order $((6t-2) + (24t-10))n = (30t-12)n$. For $H_1$, applying the energy formula for the shadow-splitting construction yields
\begin{align*}
\E(H_1) &= \sqrt{(10t-4)^2 + 4(10t-4)(20t-8)}\; \E(G) \\
&= \sqrt{(10t-4)^2 + 8(10t-4)^2}\; \E(G) \\
&= \sqrt{9(10t-4)^2}\; \E(G) \\
&= (30t-12)\, \E(G).
\end{align*}
The energy of the shadow graph $H_2$ is immediately given by the established identity
\[
\E(H_2) = (30t-12) \E(G).
\]
For the Kronecker product $H_3$, we utilize the multiplicative property of energy and the fact that $\E(K_{r,r}) = 2r$. Setting $r = 15t-6$, we obtain
\[
\E(H_3) = \E(K_{15t-6, 15t-6}) \E(G) = 2(15t-6) \E(G) = (30t-12) \E(G).
\]
Finally, for the generalized splitting graph $H_4$, we substitute the parameters $p = 6t-2$ and $q = 24t-10$ into the splitting energy formula to find
\begin{align*}
\E(H_4) &= \left( \sqrt{1 + 4(6t-2)(24t-10)} + (6t-2) - 1 \right) \E(G) \\
&= \left( \sqrt{1 + 4(144t^2 - 108t + 20)} + 6t - 3 \right) \E(G) \\
&= \left( \sqrt{576t^2 - 432t + 81} + 6t - 3 \right) \E(G) \\
&= \left( \sqrt{(24t-9)^2} + 6t - 3 \right) \E(G) \\
&= (24t - 9 + 6t - 3) \E(G) \\
&= (30t - 12) \E(G).
\end{align*}
Since the energies of all four graphs are identical and they share the same order, they are mutually equienergetic.
\end{proof}
\section{Construction of Borderenergetic Graphs}
In this section, we apply the energy formulas for the generalized splitting and shadow-splitting constructions to identify specific parameters that yield new families of borderenergetic graphs.
\begin{corollary}
For $k \ge 1$, the generalized splitting graphs $\mathcal{S}_{k+1, k}(K_3)$ and $\mathcal{S}_{9k+6, k+1}(K_3)$ are borderenergetic.
\end{corollary}
\begin{proof}
Let $H = \mathcal{S}_{p,q}(K_3)$ be a generalized splitting graph. The order of the base graph $K_3$ is $n=3$, and its energy is well known to be $\E(K_3) = 4$. The order of the constructed graph $H$ is $N = 3(p+q)$. To establish that $H$ is borderenergetic, we must demonstrate that its energy satisfies the condition $\E(H) = 2(N-1)$. 

We first consider the family $\mathcal{S}_{k+1, k}(K_3)$ with parameters $p = k+1$ and $q = k$. The order of this graph is $N = 3(k+1+k) = 6k+3$. The energy of a complete graph of order $N$ is 
\[
2(N-1) = 2(6k+3-1) = 12k+4.
\]
Applying the energy formula for the generalized splitting graph the energy of $H$ is calculated as
\begin{align*}
\E(H) &= \left( \sqrt{1 + 4(k+1)k} + (k+1) - 1 \right) \E(K_3) \\
&= \left( \sqrt{4k^2 + 4k + 1} + k \right) \cdot 4 \\
&= \left( \sqrt{(2k+1)^2} + k \right) \cdot 4 \\
&= (2k + 1 + k) \cdot 4 \\
&= (3k + 1) \cdot 4 \\
&= 12k + 4.
\end{align*}
Since $\E(H) = 2(N-1)$, the graph $\mathcal{S}_{k+1, k}(K_3)$ is borderenergetic for all $k \ge 1$.

Next, we evaluate the family $\mathcal{S}_{9k+6, k+1}(K_3)$ with parameters $p = 9k+6$ and $q = k+1$. The order of this graph is $N = 3(9k+6+k+1) = 30k+21$. The required energy for $H$ to be borderenergetic is given by
\[
2(N-1) = 2(30k+20) = 60k+40.
\]
Substituting the parameters into the generalized splitting energy formula yields
\begin{align*}
\E(H) &= \left( \sqrt{1 + 4(9k+6)(k+1)} + (9k+6) - 1 \right) \E(K_3) \\
&= \left( \sqrt{1 + 36k^2 + 60k + 24} + 9k + 5 \right) \cdot 4 \\
&= \left( \sqrt{36k^2 + 60k + 25} + 9k + 5 \right) \cdot 4 \\
&= \left( \sqrt{(6k+5)^2} + 9k + 5 \right) \cdot 4 \\
&= (6k + 5 + 9k + 5) \cdot 4 \\
&= 60k + 40.
\end{align*}
Since $\E(H) = 2(N-1)$, the graph $\mathcal{S}_{9k+6, k+1}(K_3)$ is borderenergetic for all $k \ge 1$.
\end{proof}
\begin{corollary}
For $t \ge 1$, the shadow-splitting graph $\mathcal{H}_{c,k}(K_r)$ is borderenergetic for the parameters $c = (t+1)^2, \quad k = t(2t+1) \quad \text{and} \quad r = 3t+4.$
\end{corollary}
\begin{proof}
Let $H = \mathcal{H}_{c,k}(K_r)$ be a graph of order $N = (c+k)r$. Utilizing the parameters provided, we first evaluate the order of the graph
\begin{align*}
c + k &= (t^2 + 2t + 1) + (2t^2 + t) = 3t^2 + 3t + 1, \\
N &= (3t^2 + 3t + 1)(3t + 4) = 9t^3 + 21t^2 + 15t + 4.
\end{align*}
The energy for complete graph of order $N$ is therefore
\[
2(N-1) = 2(9t^3 + 21t^2 + 15t + 3) = 18t^3 + 42t^2 + 30t + 6.
\]
We next calculate the energy of the construction using $\E(\mathcal{H}_{c,k}(G)) = \sqrt{c^2 + 4ck} \E(G)$. Since $G = K_r$, its energy is $\E(K_r) = 2(r-1) = 2(3t + 3) = 6(t+1)$. We first simplify the radical expression as follows
\begin{align*}
\sqrt{c^2 + 4ck} &= \sqrt{(t+1)^4 + 4(t+1)^2 t(2t+1)} \\
&= (t+1) \sqrt{(t+1)^2 + 8t^2 + 4t} \\
&= (t+1) \sqrt{9t^2 + 6t + 1} \\
&= (t+1)(3t + 1).
\end{align*}
Substituting these values back into the energy formula, we obtain
\begin{align*}
\E(H) &= [6(t+1)][(t+1)(3t+1)] \\
&= 6(t+1)^2(3t+1) \\
&= 6(t^2 + 2t + 1)(3t + 1) \\
&= 18t^3 + 42t^2 + 30t + 6.
\end{align*}
Since the calculated energy $\E(H)$ coincides with $2(N-1)$, we conclude that the graph family $\mathcal{H}_{c,k}(K_r)$ is borderenergetic for all $t \ge 1$.
\end{proof}
\begin{corollary}
For $t \ge 1$, let $G_t = t K_{3t+4} \cup K_{3t+5}$ be the disjoint union of \;$t$ copies of $K_{3t+4}$ and one copy of $K_{3t+5}$. Then the shadow-splitting graph $\mathcal{H}_{c,k}(G_t)$ is borderenergetic for the parameters $c = (t+1)^2$ and $k = t(2t+1)$.
\end{corollary}
\begin{proof}
The order of $G_t$ is given by
\[
n = t(3t+4) + (3t+5) = 3t^2 + 7t + 5.
\]
The energy of $G_t$ is calculated as
\begin{align*}
\E(G_t) &= t \cdot \E(K_{3t+4}) + \E(K_{3t+5}) \\
&= t[2(3t+3)] + 2(3t+4) \\
&= 6t^2 + 6t + 6t + 8 \\
&= 6t^2 + 12t + 8.
\end{align*}
Let $H = \mathcal{H}_{c,k}(G_t)$ be the shadow-splitting graph of order $N = (c+k)n$. Substituting the parameters $c$ and $k$, we find that $c+k = 3t^2 + 3t + 1$. The order of $H$ is therefore
\begin{align*}
N &= (3t^2 + 3t + 1)(3t^2 + 7t + 5) \\
&= 9t^4 + 30t^3 + 39t^2 + 22t + 5.
\end{align*}
The required energy for $H$ to be borderenergetic is $2(N-1)$, which simplifies to
\[
2(N-1) = 2(9t^4 + 30t^3 + 39t^2 + 22t + 4) = 18t^4 + 60t^3 + 78t^2 + 44t + 8.
\]
Next, we evaluate the energy of $H$ using the identity $\E(\mathcal{H}_{c,k}(G_t)) = \sqrt{c^2+4ck}\,\E(G_t)$. As established in previous sections, the radical expression for these specific values of $c$ and $k$ simplifies to $(t+1)(3t+1)$. Thus, the energy of $H$ is
\begin{align*}
\E(H) &= (t+1)(3t+1)(6t^2 + 12t + 8) \\
&= (3t^2 + 4t + 1)(6t^2 + 12t + 8) \\
&= 18t^4 + 36t^3 + 24t^2 + 24t^3 + 48t^2 + 32t + 6t^2 + 12t + 8 \\
&= 18t^4 + 60t^3 + 78t^2 + 44t + 8.
\end{align*}
Since the energy of $H$ is exactly $2(N-1)$, the graph $\mathcal{H}_{c,k}(G_t)$ is borderenergetic for all $t \ge 1$.
\end{proof}

\section{Conclusion}
This paper introduced two novel graph operations, the generalized splitting graph $\mathcal{S}_{p,q}(G)$ and the shadow-splitting graph $\mathcal{H}_{c,k}(G)$, and characterized their adjacency energies. Utilizing these constructions, we established several new families of equienergetic and borderenergetic graphs. Future investigations may extend these spectral investigations to the Laplacian and signless Laplacian matrices to further explore the structural invariants of these constructions.

\end{document}